\documentclass{amsart}
\usepackage{amsmath,amsthm,amssymb}
\usepackage{hyperref,graphicx,color,tikz}

\usepackage{soul}

\title[Proof of Chapoton's conjecture on Newton polygons]{Proof of Chapoton's conjecture on Newton polygons of $q$-Ehrhart polynomials}

\author{Jang Soo Kim}
\address{
Department of Mathematics, Sungkyunkwan University, Suwon 16420,
South Korea}
\email{jangsookim@skku.edu}

\author{U-Keun Song}
\address{
Department of Mathematics, Sungkyunkwan University, Suwon 16420,
South Korea}
\email{sukeun319@skku.edu}

\date{\today}

\thanks{This work was supported by NRF grants \#2016R1D1A1A09917506 and \#2016R1A5A1008055.}

\keywords{$q$-Ehrhart polynomial, Newton polygon, order polytope, $P$-partition}

\subjclass[2010]{Primary: 06A07; Secondary: 52B20, 05A30}

\newtheorem{thm}{Theorem}[section]
\newtheorem{lem}[thm]{Lemma}
\newtheorem{prop}[thm]{Proposition}
\newtheorem{cor}[thm]{Corollary}
\theoremstyle{definition}

\newtheorem{defn}[thm]{Definition}
\newtheorem{conj}[thm]{Conjecture}

\newcommand\Qbinom[3]{\genfrac{[}{]}{0pt}{}{#1}{#2}_{#3}}
\newcommand\qbinom[2]{\Qbinom{#1}{#2}{q}}
\newcommand\LL{\mathcal{L}}
\newcommand\NN{\mathbb{N}}
\newcommand\order{\mathcal{O}}

\newcommand\Sym{\mathfrak{S}}

\newcommand\maj{\operatorname{maj}}
\newcommand\Des{\operatorname{Des}}
\newcommand\des{\operatorname{des}}
\newcommand\DB{\operatorname{DB}}

\newcommand\mc{\operatorname{mc}}
\newcommand\PP{\mathcal{P}}
\newcommand\RR{\mathbb{R}}
\newcommand\QQ{\mathbb{Q}}
\newcommand\ZZ{\mathbb{Z}}
\newcommand\NT{\mathrm{Newton}}

\renewcommand\vec[1]{\mathbf{#1}}

\begin{document}

\begin{abstract}
Recently, Chapoton found a $q$-analog of Ehrhart polynomials, which are polynomials in $x$ whose coefficients are rational functions in $q$. 
Chapoton conjectured the shape of the Newton polygon of the numerator of the $q$-Ehrhart polynomial of an order polytope. In this paper, we prove Chapoton's conjecture.
\end{abstract}

\maketitle
% \tableofcontents

\section{Introduction}

In 1962, Ehrhart \cite{Ehrhart} discovered certain polynomials associated to lattice polytopes.  These polynomials are now widely known and called Ehrhart polynomials. They contain important information of lattice polytopes such as the number of lattice points in the polytope, the number of lattice points in the relative interior and the relative volume of the polytope.

Recently, Chapoton \cite{Chapoton2016} found a $q$-analog of Ehrhart polynomials and generalized some properties of them. A $q$-Ehrhart polynomial is a polynomial in variable $x$ whose coefficients are rational functions in $q$. Thus we can write a $q$-Ehrhart polynomial as a rational function in $q$ and $x$ whose numerator is a polynomial in $q$ and $x$, and whose denominator is a polynomial in $q$.  In the same paper, Chapoton conjectured the shape of the Newton polygon of the numerator of the $q$-Ehrhart polynomial associated to an order polytope. The goal of this paper is to prove Chapoton's conjecture.

First, we briefly review basic properties of Ehrhart polynomials and their $q$-analogs. See
\cite{Barvinok1999, Beck2007, Beck2012} for more details in Ehrhart polynomials. 

A point in $\RR^m$ is called a \emph{lattice point} if all the coordinates are integers. A \emph{lattice polytope} is a polytope whose vertices are lattice points.  All polytopes considered in this paper are  lattice polytopes.

For a polytope $M$ and an integer $n$, we denote by $nM$ the dilation of $M$ by a scale factor of $n$, i.e.,
\[
nM = \{n\vec x: \vec x\in M\}.
\]
For a lattice polytope $M$ in $\RR^m$, 
there exists a polynomial $E(x)$, called
the \emph{Ehrhart polynomial of $M$}, satisfying 
the following interesting properties:
\begin{itemize}
\item $E(n) = |nM\cap \ZZ^m|$ for all integers $n\ge0$.
\item $(-1)^{\dim M}E(-n) = |nM^\circ\cap \ZZ^m|$ for all integers $n\ge0$, where
$\dim M$ is the dimension of $M$ and $M^\circ$ is the relative interior of $M$.
\item The degree of $E(x)$ is equal to the dimension of $M$.
\item The leading coefficient of $E(x)$ is equal to the relative volume of $M$.
\end{itemize}

For a polytope $M$ in $\RR^m$, let
\[
W(M,q) = \sum_{\vec x\in M\cap \ZZ^m} q^{|\vec x|},
\]
where for $\vec x=(x_1,\dots,x_m)$ we denote
\[
|\vec x| = x_1+\dots+x_m.
\]

We use the standard notation for $q$-integers:
for $n\in\ZZ$, 
\[
[n]_q := \frac{1-q^n}{1-q},
\]  
and, for integers $n\ge k\ge 0$, 
\[
[n]_q!: = [1]_q[2]_q\dots[n]_q, \qquad
\qbinom{n}{k}: = \frac{[n]_q!}{[k]_q![n-k]_q!}.
\]
Note that for $n\ge0$ and $a,b\in\ZZ$, we have 
\begin{align*}
[n]_q &= 1+q+q^2+\dots+q^{n-1},  \\
[-n]_q &= -q^{-n}[n]_q,\\
[a+b]_q &= [a]_q+q^a [b]_q.
\end{align*}

Chapoton \cite[Theorem~3.1]{Chapoton2016} found a $q$-analog of Ehrhart polynomials as follows. 

\begin{thm}[Chapoton]\label{thm:Chapoton}
Let $M$ be a polytope satisfying  the following conditions:
\begin{itemize}
\item For every vertex $\vec x$ of $M$, we have $|\vec x|\ge0$.
\item For every edge between two vertices $\vec x$ and $\vec y$ of $M$, 
we have $|\vec x|\ne |\vec y|$.
\end{itemize}
Then there is a polynomial $E(x)\in \QQ(q)[x]$ such that for every integer $n\ge0$,
\[
E([n]_q) = W(nM,q).
\]
\end{thm}
The polynomial $E(x)$ in Theorem~\ref{thm:Chapoton} is called
the \emph{$q$-Ehrhart polynomial} of the polytope $M$. We note that in \cite{Chapoton2016}, more generally, Chapoton considers a linear form $\lambda(\vec x)$ on $\RR^m$ in place of $|\vec x|$. In this setting with a linear form, Chapoton \cite[Theorem~3.5]{Chapoton2016} also shows a nice $q$-analog of the Ehrhart-Macdonald reciprocity:
\[
E([-n]_q) =(-1)^{\dim M} W(nM^\circ,1/q).
\]

We note that Kim and Stanton \cite[Theorem~9.3]{KimStanton17} showed that the leading coefficient of the $q$-Ehrhart polynomial
of an order polytope is equal to the $q$-volume of the order polytope, which is defined as a Jackson's $q$-integral over the order polytope. 

In order to state Chapoton's conjecture we need some notation and terminology.

For a polynomial $f(x_1,\dots,x_k)$ in $x_1,\dots,x_k$, we denote by
$[x_1^{i_1}\dots x_k^{i_k}]f(x_1,\dots,x_k)$ the coefficient of $x_1^{i_1}\dots x_k^{i_k}$ in $f(x_1,\dots,x_k)$.
For a polynomial $f(x_1,\dots,x_k)$ in $x_1,\dots,x_k$, the \emph{Newton polytope} of $f(x_1,\dots,x_k)$, denoted by $\NT(f(x_1,\dots,x_k))$, is the convex hull of the points $(i_1,\dots,i_k)$ such that $[x_1^{i_1}\dots x_k^{i_k}]f(x_1,\dots,x_k)\ne 0$. In this paper, we consider \emph{Newton polygons}, which are Newton polytopes of two-variable functions. 

For a poset $P$  on $\{1,2,\dots,m\}$, the \emph{order polytope}  $\order(P)$ of $P$ is defined by
\[
\order(P)=\{(x_1,\dots,x_m)\in [0,1]^m: x_i\le x_j \mbox{ if } i\le_P j\}.
\]

As mentioned in \cite{Chapoton2016}, using the properties of vertices and edges of an order polytope in \cite{Stanley_two_poset} one can check that 
every order polytope satisfies the conditions in Theorem~\ref{thm:Chapoton}. Therefore, we can consider the $q$-Ehrhart polynomial of an order polytope. 

Let $E_P(x)$ be the $q$-Ehrhart polynomial of $\order(P)$. We denote by $N_P(q,x)$ be the numerator of $E_P(x)$. More precisely, $N_P(q,x)$ is the unique polynomial in $\ZZ[q,x]$ with positive leading coefficient such that
\[
E_P(x) = \frac{N_P(q,x)}{D(q)},
\]
for some polynomial $D(q)\in \ZZ[q]$ with $\gcd(N_P(q,x),D(q))=1$.

For integers $1\le a_1\le a_2\le \dots\le a_m$ and $h\ge a_1+\dots+a_m$, we define $C(a_1,\dots,a_m;h)$ to be the convex hull of the points $(0,0)$, $(a_1+\dots+a_i,i)$ for $1\le i\le m$, $(h,m)$ and $(h-m,0)$. See Figure~\ref{fig:C} for an example.

\begin{figure}
  \centering
\begin{tikzpicture}[scale=0.5]
\draw[step=1cm,gray,very thin] (0,0) grid (10.5,4.5); 
\foreach \x in {1,...,10}
\draw (\x cm,1pt) -- (\x cm,-1pt) node[anchor=north] {$\x$};
\foreach \y in {1,...,4}
\draw (1pt, \y cm) -- (-1pt, \y cm) node[anchor=east] {$\y$};
\draw[->,thick] (-.5,0) -- (11.5,0) node[anchor=north] {$q$};
\draw[->,thick] (0,-.5) -- (0,5.5) node[anchor=north west] {$x$};
\draw [fill=lightgray,very thick] (0,0) -- (1,1) -- (3,2) -- (5,3) -- (8,4) -- (10,4) -- (6,0) -- (0,0);
\draw (0,0) node[anchor=north east] {$0$};
\end{tikzpicture}
\caption{The polygon $C(1,2,2,3;10)$ in the $(q,x)$-coordinate system.}
\label{fig:C}  
\end{figure}
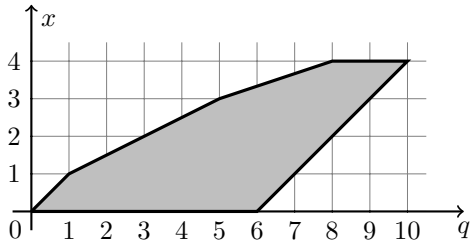

Let $P$ be a poset and $x\in P$.
A \emph{chain ending at $x$} (resp.~\emph{starting at $x$}) is a subset $\{t_1<_P \dots <_P t_k\}$
of $P$ with $t_k=x$ (resp.~$t_1=x$). The \emph{size} of a chain is the number of elements in the chain. 
 We denote by $\mc_P(x)$ the maximum size of a chain ending at $x$. We also denote by $\overline{\mc}_P(x)$ the maximum size of a chain starting at $x$. When there is no possible confusion, we will simply write as $\mc(x)$ and $\overline{\mc}(x)$ instead of $\mc_P(x)$ and $\overline{\mc}_P(x)$.

In \cite[Conjecture~5.3]{Chapoton2016}, Chapoton proposed the following conjecture on the shape of the Newton polygon of $N_P(q,x)$.

\begin{conj}\label{conj:chapoton}
  Let $P$ be a poset on $\{1,2,\dots,m\}$. Suppose that $a_1\le a_2\le \dots\le a_m$
is the increasing rearrangement of $\overline{\mc}(1),\dots,\overline{\mc}(m)$.  Then the Newton polygon of the numerator of the $q$-Ehrhart polynomial of $\order(P)$ is given by \[ \NT(N_P(q,x))=C(a_1,\dots,a_m;h), \] for some integer $h\ge a_1+\dots+a_m$. 
 \end{conj}

 The goal of this paper is to prove Conjecture~\ref{conj:chapoton}.  As Chapoton points out in \cite{Chapoton2016}, the $q$-Ehrhart polynomial $E_P(x)$ of $\order(P)$ can be understood as a generating function for $P$-partitions of $\overline{P}$, the dual poset of $P$. It is well-known that the generating function for $P$-partitions can be expressed in terms of linear extensions of the poset. One of the main ingredients of our proof of Conjecture~\ref{conj:chapoton} is Corollary~\ref{cor:mc}, which gives a description of the minimum of $\maj(\pi)-k \des(\pi)$ over all linear extensions $\pi$ of $P$.

The rest of this paper is organized as follows.  In Section~\ref{sec:main-result}, we recall necessary definitions and state our main result (Theorem~\ref{thm:main}), which describes the precise shape of the Newton polygon of $[m]_q!E_{P}(x)$. Then we show that Theorem~\ref{thm:main} implies Conjecture~\ref{conj:chapoton}. In Section~\ref{sec:some-prop-posets} we find some property of the linear extensions of a poset. In Section~\ref{sec:proof-theor-refthm:m} we prove Theorem~\ref{thm:main}.

\section{The main result}
\label{sec:main-result}

In this section we state our main theorem, which implies Conjecture~\ref{conj:chapoton}. 

We first recall some definitions on permutations and posets. We refer the reader to \cite{EC1} for more details.

The set of nonnegative integers is denoted by $\NN$.

Let $\Sym_m$ be the set of permutations of $\{1,2,\dots,m\}$. For $\pi=\pi_1\dots\pi_m\in\Sym_m$,
a \emph{descent} of $\pi$ is an integer $1\le i\le m-1$ such that $\pi_i>\pi_{i+1}$. We denote by $\Des(\pi)$ the set of descents of $\pi$. We define $\maj(\pi)=\sum_{i\in\Des(\pi)}i$ and $\des(\pi)=|\Des(\pi)|$. 

Let $P$ be a poset on $\{1,2,\dots,m\}$. A \emph{$P$-partition} is an order-reversing map $\sigma:P\to\NN$, i.e., $\sigma(x)\ge \sigma(y)$ if $x\le_P y$. For a $P$-partition $\sigma$, let $|\sigma|=\sigma(1)+\dots+\sigma(m)$.
 We denote by $\PP(P)$ the set of $P$-partitions.  For an integer $n$, we denote by $\PP(P,n)$ the set of $P$-partitions $\sigma$ satisfying $\sigma(x)\le n$ for all $x\in P$.

 We say that $P$ is \emph{naturally labeled} if $x\le_P y$ implies $x\le y$.  A \emph{linear extension} of $P$ is a permutation $\pi=\pi_1\dots\pi_m\in\Sym_m$ such that $\pi_i\le_P\pi_j$ implies $i\le j$. We denote by $\LL(P)$ the set of linear extensions of $P$.  
Note that if $P$ is naturally labeled, $\LL(P)$ always contains the identity permutation.

We need the following lemma, which gives a connection between certain generating functions for $\PP(P,n)$ and $\LL(P)$. 

\begin{lem}\label{lem:qbinom}
For a naturally labeled poset $P$ on $\{1,2,\dots,m\}$, we have  
\[
\sum_{\sigma\in \PP(P,n)} q^{|\sigma|}
=\sum_{\pi\in\LL(P)}q^{\maj(\pi)} \qbinom{n-\des(\pi)+m}{m}.
\]
\end{lem}
\begin{proof}
For a permutation $w\in\Sym_m$, let $S_w$ denote the set of all functions $f:P\to\NN$ satisfying the following conditions:
\begin{itemize}
\item $f(w_1)\ge f(w_2)\ge \cdots\ge f(w_m)$ and 
\item $f(w_i)>f(w_{i+1})$ if $i\in\Des(w)$. 
\end{itemize}
It is well known \cite[Lemma~3.15.3]{EC1} that
\[
\PP(P) = \biguplus_{\pi\in\LL(P)} S_\pi.
\]
Let $S_\pi(n)=S_\pi\cap \PP(P,n) $. Then we have
\[
\PP(P,n) = \biguplus_{\pi\in\LL(P)} S_\pi(n).
\]
Thus, 
\[
\sum_{\sigma\in \PP(P,n)} q^{|\sigma|} = \sum_{\pi\in\LL(P)}\sum_{\sigma\in S_\pi(n)}q^{|\sigma|} =
\sum_{\pi\in\LL(P)}\sum_{\substack{n\ge i_1\ge \cdots\ge i_m\ge0\\ i_j>i_{j+1} \:\:\: {\rm if} \:\:\: j\in\Des(\pi)}} q^{i_1+\dots+i_m}.
\]
It is shown in \cite[Lemma~4.5]{KimStanton17} that
\[
\sum_{\substack{n\ge i_1\ge \cdots\ge i_m\ge0\\ i_j>i_{j+1} \:\:\: {\rm if} \:\:\: j\in\Des(\pi)}} q^{i_1+\dots+i_m}
=q^{\maj(\pi)} \qbinom{n-\des(\pi)+m}{m}, 
\]
which completes the proof.
\end{proof}

For a poset $P$, we denote its dual by $\overline{P}$, that is, $x\le_P y$ if and only if $y\le_{\overline{P}} x$.
By definition, for a poset $P$ and an integer $n\in\NN$, we have
\begin{equation}
  \label{eq:W}
W(n\order(P),q) = \sum_{\sigma\in \PP(\overline{P},n)} q^{|\sigma|}.  
\end{equation}
Therefore, the $q$-Ehrhart polynomial $E_P(x)$ of $\order(P)$ is closely related to $P$-partitions of $\overline{P}$.
The next proposition shows that $E_P(x)$ can be written as a generating function for linear extensions of $\overline{P}$.

\begin{prop}\label{prop:Ep}
Let $P$ be a poset on $\{1,2,\dots,m\}$. Suppose that $\overline{P}$ is naturally labeled. Then
the $q$-Ehrhart polynomial of $\order(P)$ is
\[
E_P(x)=\frac{1}{[m]_q!}\sum_{\pi\in\LL(\overline{P})} q^{\maj(\pi)}
\prod_{i=1}^m([i-\des(\pi)]_q+  q^{i-\des(\pi)} x).
\]
\end{prop}
\begin{proof}
Let $f(x)$ be the right hand side. Then 
\[
f([n]_q)=\sum_{\pi\in\LL(\overline{P})} q^{\maj(\pi)}
\frac{\prod_{i=1}^m[i-\des(\pi)+n]_q}{[m]_q!}=
\sum_{\pi\in\LL(\overline{P})}q^{\maj(\pi)} \qbinom{n-\des(\pi)+m}{m}.
\]
On the other hand, by Lemma~\eqref{lem:qbinom} and \eqref{eq:W}, we have
\[
W(n\order(P),q) =\sum_{\pi\in\LL(\overline{P})}q^{\maj(\pi)} \qbinom{n-\des(\pi)+m}{m}.
\]
Thus $f([n]_q)=W(n\order(P),q)$ for all $n\in\NN$ and we obtain
$E_P(x)=f(x)$. 
\end{proof}

Now we define a polynomial $F_P(q,x)$ in $q$ and $x$, which will be used throughout this paper. 

\begin{defn}
For a poset $P$ on $\{1,2,\dots,m\}$, we define
\[
F_P(q,x)=\sum_{\pi\in\LL(P)} q^{\maj(\pi)}
\prod_{i=1}^m([i-\des(\pi)]_q+  q^{i-\des(\pi)} x).
\]
\end{defn}

Note that we always have $F_P(q,x)\in \ZZ[q,x]$ because for every $\pi\in\LL(P)$, the power of $q$ in each summand is at least
\[
\maj(\pi)+\sum_{i=1}^{\des(\pi)}(i-\des(\pi)) \ge \binom{\des(\pi)+1}2 +\binom{\des(\pi)+1}2 - \des(\pi)^2\ge0 .
\]
Proposition~\ref{prop:Ep} implies that for a naturally labeled poset $P$ on $\{1,2,\dots,m\}$, we have
\begin{equation}
  \label{eq:Fp}
F_P(q,x)=[m]_q!E_{\overline{P}}(x).
\end{equation}

\begin{prop}\label{prop:parallelogram}
  Let $P$ be a poset on $\{1,2,\dots,m\}$ such that $\overline{P}$ is naturally labeled. Suppose that $a_1\le a_2\le \dots\le a_m$
is the increasing rearrangement of $\overline{\mc}(1),\dots,\overline{\mc}(m)$.
Then we have
\[ 
\NT(N_P(q,x))=C(a_1,\dots,a_m;h), 
\]
for some $h\ge a_1+\dots+a_m$ if and only if
\[ 
\NT(F_{\overline{P}}(q,x))=C(a_1,\dots,a_m;h'),
\]
for some $h'\ge a_1+\dots+a_m$. Moreover, in this case we always have $h'=h+r$, where
$r=\deg\phi(q)$ and $\phi(q)=\gcd(F_{\overline{P}}(q,x),[m]_q!)$.
\end{prop}
\begin{proof}
By \eqref{eq:Fp}, we have
\[
F_{\overline{P}}(q,x) = N_P(q,x) \phi(q).
\]  
Since $\phi(q)$ divides $[m]_q!$, the leading coefficient and the constant term of $\phi(q)$ are both $1$. Thus, we have
\[
\phi(q) = q^r+c_{r-1}q^{r-1}+\dots+c_1q^1+1,
\]
for some $c_1,\dots,c_{r-1}\in \ZZ$.
Hence, for each $1\le k\le m$, we have
\[
\max\{i:[q^ix^k]N_{P}(q,x)\ne 0\} 
=\max\{i:[q^{i+r}x^k]F_{\overline{P}}(q,x)\ne 0\},
\]
\[
\min\{i:[q^ix^k]N_{P}(q,x)\ne0\} 
=\min\{i:[q^{i}x^k]F_{\overline{P}}(q,x)\ne0\},
\]
which imply the statement.
\end{proof}

Now we state our main theorem.

\begin{thm}\label{thm:main}
  Let $P$ be a naturally labeled poset on $\{1,2,\dots,m\}$. 
Let $b_1\le b_2\le\dots\le b_m$ be the increasing rearrangement of $\mc(1),\mc(2),\dots,\mc(m)$.
Then the Newton polygon of 
\[
F_P(q,x)=[m]_q!E_{\overline{P}}(x)=\sum_{\pi\in\LL(P)} q^{\maj(\pi)}
\prod_{i=1}^m([i-\des(\pi)]_q+  q^{i-\des(\pi)} x)
\]
is given by
\[
\NT(F_{P}(q,x))=C\left(b_1,\dots,b_m;\binom{m+1}2\right).
\]
\end{thm}

We prove Theorem~\ref{thm:main} in Section~\ref{sec:proof-theor-refthm:m}. Note that in Theorem~\ref{thm:main} we
 have
\[ b_1+\dots+b_m \le \binom{m+1}2,  \]
which follows from the fact that $b_1=1$ and $b_{i+1}\le b_i+1$ for all $i$. 

We finish this section by showing that Theorem~\ref{thm:main} implies Conjecture~\ref{conj:chapoton}.

\begin{proof}[Proof of Conjecture~\ref{conj:chapoton}]
Note that relabeling of $P$ does not affect $E_P(q,x)$. Hence, we can assume that $\overline{P}$ is naturally labeled.   
Observe that $\overline{\mc}_P(x)=\mc_{\overline{P}}(x)$ for all $x\in \{1,2,\dots,m\}$. By Theorem~\ref{thm:main}, 
\[
\NT(F_{\overline{P}}(q,x))=C\left(a_1,\dots,a_m;\binom{m+1}2\right).
\]
By Proposition~\ref{prop:parallelogram}, we obtain that
\[
\NT(N_P(q,x))=C(a_1,\dots,a_m;h),
\]
for some integer $h\ge a_1+\dots+a_m$. This completes the proof.
\end{proof}

\section{Some properties of linear extensions}
\label{sec:some-prop-posets}

In this section we prove some properties of posets which will be used in the next section.

\begin{lem}\label{lem:resolve}
Let $P$ be a naturally labeled poset on $\{1,2,\dots,m\}$ and 
$\pi\in\LL(P)$. Suppose that $\Des(\pi)\ne\emptyset$ and  $c$ is the largest descent of $\pi$.
Then there is a permutation $\sigma\in\LL(P)$ such that $\Des(\sigma)=\Des(\pi)\setminus\{c\}$. 
\end{lem}
\begin{proof}
  Let $j$ be the second largest descent of $\pi$. If $c$ is the only descent of $\pi$, we set $j=0$.  Let $\sigma=\pi_1\dots \pi_j \pi_{j+1}' \dots \pi_m'$, where $\pi_{j+1}'<\cdots<\pi_m'$ is the increasing rearrangement of $\pi_{j+1},\dots,\pi_m$.  

We claim that $\sigma\in\LL(P)$.  Consider two elements $x$ and $y$ with $x<_P y$.  If $y$ is in $\pi_1\dots\pi_j$, since $\pi$ is a linear extension, $x$ must appear before $y$. 
Otherwise $y$ is in the increasing sequence $\pi_{j+1}' \dots \pi_m'$. Since $P$ is naturally labeled, $x<y$ and
$x$ cannot appear after $y$ in $\pi_{j+1}' \dots \pi_m'$. Thus we always have $x$ before $y$ in $\sigma$.

Since $\pi_j>\pi_{j+1}\ge \pi_{j+1}'$, we have
$\Des(\sigma)=\Des(\pi)\setminus\{c\}$. 

\end{proof}

\begin{defn}
Let $\pi=\pi_1\dots\pi_m\in \Sym_m$. A \emph{descent block} of $\pi$ is a maximal consecutive subsequence of $\pi$ which is in decreasing order. We denote by $\DB_i(\pi)$ the set of elements in the $i$th descent block of $\pi$. 
\end{defn}

For example, if $\pi=384196725$, then the descent blocks of $\pi$ are
$3$, $841$, $96$, $72$, $5$, and $\DB_1(\pi) = \{3\}$, 
$\DB_2(\pi) = \{1,4,8\}$, 
$\DB_3(\pi) = \{6,9\}$,
$\DB_4(\pi) = \{2,7\}$, $\DB_5(\pi) = \{5\}$.

\begin{lem}\label{lem:partition}
Fix  integers $r,p,k,m$ and mutually disjoint subsets $C_1,\dots,C_p$ of $\{1,2,\dots,m\}$ such that
\[
|C_1|+\dots+|C_{p-1}|\le k<|C_1|+\dots+|C_{p}|.
\]
Let $u=|C_1|+\dots+|C_{p}|-k$. Then, for every $r$-tuples $B=(B_1,\dots,B_r)$ of mutually disjoint subsets of $\{1,2,\dots,k\}$
satisfying $B_i\subseteq C_1\cup\dots \cup C_i$ for $1\le i\le r$, we have
\begin{equation}
  \label{eq:BC}
|B_1|+2|B_2|+\dots+r|B_r|\ge |C_1|+2|C_2|+\dots+(p-1)|C_{p-1}|+p(|C_p|-u).  
\end{equation}
Moreover, the equality holds if and only if the following conditions hold: $r=p$ and
$B_i=C_i$ for $1\le i\le r-1$ and $B_r\subseteq C_r$.
\end{lem}
\begin{proof}
First assume that the conditions for the equality hold. Then
\[
\sum_{i=1}^r i |B_i|= \sum_{i=1}^{r-1} i |C_i|+r|B_r|.
\]
Since $p=r$, we have $|B_r|=|C_p|-u$ and the equality of \eqref{eq:BC} holds. 

Now suppose that  $B=(B_1,\dots,B_r)$ is an $r$-tuple of mutually disjoint subsets of $\{1,2,\dots,k\}$ 
satisfying $B_i\subseteq C_1\cup\dots \cup C_i$ for $1\le i\le r$ that minimizes the value
$\sum_{i=1}^r i |B_i|$. It suffices to show that $B$ satisfies the conditions for the equality. 
For a contradiction, suppose that the conditions do not hold. 
Then there are two cases.

Case 1: We can find the smallest $1\le s\le r-1$ such that $B_s\ne C_s$. By the assumption on $B$,
there is an element $x\in B_s\cap C_j$ for some $1\le j\le s-1$. Define
$B' =(B'_1,\dots,B'_r)$ to be the $r$-tuple obtained from $B$ by replacing
$B_j$ by $B_j\cup\{x\}$ and $B_s$ by $B_s\setminus\{x\}$. Then $\sum_{i=1}i|B'_i|<\sum_{i=1}i|B_i|$, which is a contradiction.

Case 2: $\DB_i(\pi)\ne C_i$ for $1\le i\le r-1$ and $\DB_r(\pi)\not\subseteq C_r$. Similarly, we can find
an element $x\in B_r\cap C_j$ for some $1\le j\le r-1$ and obtain a contradiction. 

By the above two cases, $B$ must satisfy the conditions for the equality. This finishes the proof.
\end{proof}

The following proposition is the key ingredient for proving Chapoton's conjecture. 

\begin{prop}\label{prop:mc}
Let $P$ be a naturally labeled poset on $\{1,2,\dots,m\}$. 
Suppose that $b_1\le b_2\le\dots\le b_m$ is the increasing rearrangement of $\mc(1),\mc(2),\dots,\mc(m)$ and
\[
C_i=\{x\in P: \mc(x)=i\}.
\]
Then, for $\pi\in\LL(P)$ and $0\le k\le m$, we have
\[
\maj(\pi)-k\des(\pi)+\binom{k+1}2 \ge b_1+\dots+b_k.
\]
The equality holds if and only if all of the following conditions hold:
\begin{itemize}
\item $\Des(\pi)\subseteq \{1,2,\dots,k\}$,
\item $\DB_i(\pi)=C_i$, for $1\le i\le p-1$,
\item $\DB_p(\pi)\subseteq C_p$,
\end{itemize}
where $p$  is the integer satisfying
\[
|C_1|+\dots+|C_{p-1}| < k \le |C_1|+\dots+|C_{p}|.
\]
Furthermore, for every $0\le k\le m$, there is a permutation in $\LL(P)$ satisfying these conditions.
\end{prop}
\begin{proof}
For $\pi\in \Sym_m$, let
\[
b(\pi) = \maj(\pi)-k\des(\pi)+\binom{k+1}2 .
\]

Suppose that $\pi$ is a permutation in $\LL(P)$ such that $b(\pi)$ is the smallest. If $\pi$ has a descent greater than $k$, by Lemma~\ref{lem:resolve}, we can find $\sigma\in\LL(P)$ with $\Des(\sigma)=\Des(\pi)\setminus\{\max(\Des(\pi))\}$.  Then $b(\sigma)<b(\pi)$, which is a contradiction.  If $k$ is the largest descent of $\pi$, by the same construction, we can remove the descent $k$ without changing $b(\pi)$. Therefore we can assume that all descents of $\pi$ are at most $k-1$.

Let $1\le t_1<t_2<\dots<t_r= k$ be the non-descents of $\pi$ among $1,2,\dots,k$, i.e., 
\[
\{t_1,t_2,\dots,t_r\} = \{1,2,\dots,k\} \setminus \Des(\pi).
\]
Then
\begin{align*}
b(\pi)  &= \binom {k+1}2 -(t_1+\dots+t_r)-k(k-r)+\binom{k+1}2\\
&= k(r+1) -(t_1+\dots+t_r).
\end{align*}
Note that $\DB_i(\pi)=\{\pi_{t_{i-1}+1},\dots,\pi_{t_i}\}$ for $1\le i\le r$, where $t_0=0$, and
\begin{equation}
  \label{eq:DBunion}
\DB_1(\pi)\cup \cdots \cup \DB_r(\pi) = \{1,2,\dots,k\}.
\end{equation}
Since $t_i=|\DB_1(\pi)|+\dots+|\DB_i(\pi)|$ for $1\le i\le r$ and $k=|\DB_1(\pi)|+\dots+|\DB_r(\pi)|$, we have
\begin{equation}
  \label{eq:b}
b(\pi)=|\DB_1(\pi)|+2|\DB_2(\pi)|+\dots+r|\DB_r(\pi)|.  
\end{equation}

We claim that $\mc(x)\le i$ for all $x\in \DB_i(\pi)$. To prove this let $\mc(x)=\ell$ and
$x_1<_P \dots <_P x_\ell = x$ be a maximal chain ending at $x$. Since $\pi$ is a linear extension of $P$, 
$x_1,\dots,x_\ell$ must occur in this order in $\pi$. Since $P$ is naturally labeled,
we have $x_1<\dots<x_\ell$. Hence each $\DB_j(\pi)$ has at most one element among $x_1,\dots,x_\ell$. This settles the claim.

The above claim implies that 
\begin{equation}
  \label{eq:DB}
\DB_i(\pi)\subseteq C_1\cup\dots\cup C_i.   
\end{equation}

Let $u=|C_1|+\dots+|C_{p}|-k$. By Lemma~\ref{lem:partition}, we have
\begin{equation}
  \label{eq:3}
b(\pi)\ge |C_1|+2|C_2|+\dots+(p-1)|C_{p-1}|+p(|C_p|-u),  
\end{equation}
where the equality holds if and only if the following conditions hold: $r=p$ and
$\DB_i(\pi)=C_i$ for $1\le i\le p-1$ and $\DB_p(\pi)\subseteq C_p$. 

Now it remains to show that there is $\pi\in\LL(P)$ satisfying the conditions for the equality.  We construct such a permutation as follows.  Let $T$ be any subset of $C_p$ such that $k=|C_1|+\dots+|C_{p-1}|+|T|$.  Let $\pi=\pi_1\dots\pi_m\in\Sym_m$ be the permutation obtained from the empty sequence by appending the elements of $C_i$ in decreasing order for $i=1,2,\dots,p-1$, the elements of $T$ in decreasing order, and the remaining integers in increasing order. Then $\pi$ satisfying the conditions for the equality. For a contradiction, suppose that $\pi\not\in\LL(P)$. Then there are two elements $x,y\in P$ such that $x<_P y$ and $y$ appears to the left of $x$ in $\pi$.  Let $\mc(x)=i$ and $\mc(y)=j$. Then $x\in C_i$ and $y\in C_j$. Since $x<_Py$, we have $i<j$ and $x<y$. If $y$ is in $\pi_1\dots\pi_k$, then $j\le p$. In this case we have $x\in C_i=\DB_i(\pi)$ and $y\in C_j=\DB_j(\pi)$, which is a contradiction to the assumption that $y$ appear to the left of $x$. If $y$ is not in $\pi_1\dots\pi_k$, we have both $x$ and $y$ in $\pi_{k+1}\dots\pi_m$. Since these elements are in increasing order, we cannot have $y$ to the left of $x$. Therefore we must have $\pi\in\LL(P)$, which finishes the proof.
\end{proof}

The following corollary is an immediate consequence of Proposition~\ref{prop:mc}. 

\begin{cor}\label{cor:mc}
Let $P$ be a naturally labeled poset on $\{1,2,\dots,m\}$. 
Suppose that $b_1\le b_2\le\dots\le b_m$ is the increasing rearrangement of $\mc(1),\mc(2),\dots,\mc(m)$.
Then, for $0\le k\le m$, we have
\[
\min\left\{\maj(\pi)-k\des(\pi)+\binom{k+1}2: \pi\in\LL(P)\right\}
=b_1+\dots+b_k.
\]
Moreover, for $1\le k\le m$, if $P$ is not a chain, we have
\[
\min\left\{\maj(\pi)-k\des(\pi)+\binom{k+1}2: \pi\in\LL(P),
1\le \des(\pi)\le k\right\} =b_1+\dots+b_k.
\]
\end{cor}

Note that Corollary~\ref{cor:mc} allows us to find the minimum of $\maj(\pi)-k \des(\pi)$ over  $\pi\in\LL(P)$. 
The second part of Corollary~\ref{cor:mc} means that if $1\le k\le m$ and $P$ is not a chain, 
the minimum of $\maj(\pi)-k\des(\pi)+\binom{k+1}2$ for all $\pi\in\LL(P)$ is attained for $\pi\in\LL(P)$ satisfying
$1\le \des(\pi)\le k$. This will be used in the next section.

\section{Proof of Theorem~\ref{thm:main}}
\label{sec:proof-theor-refthm:m}

In this section we assume that $P$ is a naturally labeled poset on $\{1,2,\dots,m\}$
and $b_1\le b_2\le \dots\le b_m$ is the increasing rearrangement of $\mc(1),\dots,\mc(m)$.

For a polynomial $f(q)$ in $q$, define
\begin{align*}
q_{\max}(f(q)) &= \max\{i:[q^i]f(q) \ne 0\},\\
q_{\min}(f(q)) &= \min\{i:[q^i]f(q) \ne 0\}.
\end{align*}
When $f(q)=0$, we use the following convention:
\[
q_{\max}(0)=-\infty, \qquad q_{\min}(0)=\infty.
\]

Recall that
\begin{align*}
F_P(q,x) &=\sum_{\pi\in\LL(P)} q^{\maj(\pi)}
\prod_{i=1}^m([i-\des(\pi)]_q+  q^{i-\des(\pi)} x)\\
&=\sum_{s=0}^{m-1} \sum_{\pi\in\LL(P), \des(\pi)=s} q^{\maj(\pi)}
\prod_{i=1}^m( q^{i-s} x + [i-s]_q).
\end{align*}
Since $P$ is naturally labeled, $\LL(P)$ contains the identity permutation. Therefore, 
\begin{equation}
  \label{eq:FAB}
F_P(q,x)=A+B,  
\end{equation}
where
\begin{align*}
A &=\prod_{i=1}^m( q^{i} x + [i]_q),\\
B &=x\sum_{s=1}^{m-1} \sum_{\pi\in\LL(P), \des(\pi)=s} q^{\maj(\pi)-\binom s2}
\prod_{i=1}^{s-1}(x - [i]_q)
\prod_{i=1}^{m-s}( q^{i} x + [i]_q).
\end{align*}
Since $[x^0]F_P(q,x)=[x^0]A=[m]_q!$, we have
\[
q_{\max}([x^0]F_P(q,x)) = \binom m2,\qquad q_{\min}([x^0]F_P(q,x)) = 0.
\]
Therefore, in order to prove Theorem~\ref{thm:main}, it suffices to show the following two propositions.

\begin{prop}\label{prop:max}
For $1\le k\le m$, we have
\[
q_{\max}([x^k]F_P(q,x)) = \binom m2 +k.
\]  
\end{prop}

\begin{prop}\label{prop:min}
For $1\le k\le m$, we have
\[
q_{\min}([x^k]F_P(q,x)) = b_1+\dots+b_k.
\]  
\end{prop}

\begin{proof}[Proof of Proposition~\ref{prop:max}]
By \eqref{eq:FAB}, it is enough to show that
\begin{equation}
  \label{eq:Amax}
q_{\max}([x^k]A) = \binom m2 +k,
\end{equation}
\begin{equation}
  \label{eq:Bmax}
q_{\max}([x^k]B) < \binom m2 +k.
\end{equation}
In order to get the largest power of $q$, when we expand the product in $A$, we 
must select $q^ix$ or $q^{i-1}$. This implies \eqref{eq:Amax}. 

To prove \eqref{eq:Bmax}, consider $\pi\in\Sym_m$ with $\des(\pi)=s\ge1$.
Then 
\begin{align*}
&q_{\max}\left([x^{k-1}] q^{\maj(\pi)-\binom s2}
\prod_{i=1}^{s-1}(x - [i]_q) \prod_{i=1}^{m-s}( q^{i} x + [i]_q)
\right)\\
&\le \sum_{i=1}^s(m-i)-\binom s2 + 
\sum_{i=1}^{s-1}(i-1) +\sum_{i=1}^{m-s}(i-1) + (k-1)\\
&= \binom m2 -(s-1)+(k-1) < \binom m2 +k.
\end{align*}
Therefore, we obtain~\eqref{eq:Bmax}. 
\end{proof}

The rest of this section is devoted to proving Proposition~\ref{prop:min}.  

For $\pi\in\Sym_m$ with $\des(\pi)=s\ge1$ and an integer $1\le k\le m$, let
\begin{equation}
  \label{eq:t}
t(\pi,k)=[x^{k-1}] q^{\maj(\pi)-\binom s2}
\prod_{i=1}^{s-1}(x - [i]_q) \prod_{i=1}^{m-s}( q^{i} x + [i]_q).
\end{equation}
Then we always have
\begin{equation}
  \label{eq:maj-s2}
q_{\min}(t(\pi,k)) \ge \maj(\pi)-\binom{\des(\pi)}2.  
\end{equation}

We need the following two lemmas.

\begin{lem}\label{lem:qmin}
Let $\pi\in\Sym_m$ with $\des(\pi)=s\ge1$. Then,
for $s\le k\le m$, we have
\[
q_{\min}(t(\pi,k))  =\maj(\pi)-ks+\binom{k+1}2.
\]
\end{lem}
\begin{proof}
It is easy to see that the smallest power of $q$ in $[x^{k-1}]t(\pi,k)$ 
is obtained if and only if we select $x^{s-1}$ in the first product
and $q^{1+2+\dots+(k-s)}x^{k-s}$ in the second product in \eqref{eq:t}. 
Thus, we obtain 
\[
q_{\min}(t(\pi,k)) = \maj(\pi)-\binom s2+\binom{k-s+1}2, 
\]
which is equivalent to the lemma.
\end{proof}

\begin{lem}\label{lem:qmin2}
Let $P$ be a naturally labeled poset on  $\{1,2,\dots,m\}$. 
Suppose that  $P$ is not a chain. Then, for $1\le k\le m$, we have
\[
q_{\min}([x^{k}]B) = \min\left\{
\maj(\pi)-k\des(\pi)+\binom{k+1}2 : \pi\in\LL(P), 1\le \des(\pi)\le k 
\right\}.
\]
\end{lem}
\begin{proof}
Let
\begin{align*}
m_1 &= \min\{q_{\min}(t(\pi,k)): \pi\in\LL(P), 1 \le \des(\pi)\le k \},\\
m_2 &=  \min\{q_{\min}(t(\pi,k)) : \pi\in\LL(P), \des(\pi)\ge k+1 \}.
\end{align*}

By \eqref{eq:maj-s2}, for $\pi\in\LL(P)$ with $\des(\pi)\ge k+1$, we have
\[
q_{\min}(t(\pi,k)) \ge \maj(\pi)-\binom{\des(\pi)}2.  
\]
By Lemma~\ref{lem:resolve}, we can find 
$\sigma\in \LL(P)$ such that $\Des(\sigma)=\Des(\pi)\setminus \{\max(\Des(\pi))\}$.
Since $\max(\Des(\pi))\ge \des(\pi)$, we have
\[
\maj(\pi)-\binom{\des(\pi)}2 \ge 
\maj(\sigma)+\des(\pi)-\binom{\des(\pi)}2
>\maj(\sigma)-\binom{\des(\sigma)}2.
\]
By repeating this argument, we obtain that
\begin{equation}
  \label{eq:1}
\maj(\pi)-\binom{\des(\pi)}2 
>\maj(\tau)-\binom{\des(\tau)}2,
\end{equation}
for some $\tau\in\LL(P)$ with $\des(\tau)=k$. 

On the other hand, by Lemma~\ref{lem:qmin}, we have
\begin{equation}
  \label{eq:2}
q_{\min}([x^{k-1}]t(\tau,k))=\maj(\tau)-k^2+\binom{k+1}2=
\maj(\tau)-\binom{\des(\tau)}2.
\end{equation}

By \eqref{eq:1} and \eqref{eq:2}, we have $m_1 < m_2$. 
Therefore $q_{\min}([x^{k}]B) = m_1$. By applying Lemma~\ref{lem:qmin} to $m_1$, 
we obtain the desired identity.
\end{proof}

Now we give a proof of Proposition~\ref{prop:min}.

\begin{proof}[Proof of Proposition~\ref{prop:min}]
First, observe that
\[
q_{\min}([x^k]A) = \binom{k+1}2.
\]

If $P$ is a chain, then the identity permutation is the only linear extension of $P$. In this case $B=0$ and $b_i=i$. Thus 
\[
q_{\min}([x^k]F_P(q,x)) = q_{\min}([x^k]A) =\binom{k+1}2=
b_1+\dots+b_k.
\]

Now suppose that $P$ is not a chain. By Lemma~\ref{lem:qmin2} and Corollary~\ref{cor:mc} we have
\[
q_{\min}([x^k]B) =b_1+\dots+b_k \le  \binom{k+1}2.
\]
Therefore we also obtain
\[
q_{\min}([x^k]F_P(q,x)) = b_1+\dots+b_k. \qedhere
\]
\end{proof}

\section*{Acknowledgement}

The authors would like to thank the anonymous referee for his or her careful reading and helpful comments.

\end{document}